\documentclass{article}
\usepackage{amssymb}
\usepackage{amsmath}
\title{ An Order Theoretic Approach to  the   Banach Contraction Principle in  Modular  Spaces \\[0.3cm]}

\author{{Kourosh Nourouzi  \thanks {e-mail: nourouzi@kntu.ac.ir; fax: +98 21
22853650}}\\[0.4cm]
{  Department of Mathematics,  K. N. Toosi University of Technology,}\\
{\em P.O. Box 16315-1618, Tehran, Iran.}\\}

\newenvironment{proof}{\noindent {\em {Proof .}}}{$\square$
\medskip}

\newtheorem{defin}{Definition}

\newtheorem{thm}{Theorem}

\begin{document}

\maketitle \begin{abstract} In the present note,  the Banach
contraction principle is proved in complete modular spaces  via an order theoretic approach.

\end{abstract}

\renewcommand{\baselinestretch}{1.1}
\def\thefootnote{ \ }

\footnotetext{{\em} $2010$ Mathematics Subject Classification.
Primary: 47H10, 46E30
}
\vspace{1.0cm}
\section{Introduction}
A modular formulation of Banach contraction principle was given
in \cite{khamsi, han} under the
 superfluous assumptions $\Delta_2$-condition and $s$-convexity as follows:

\begin{thm}{\rm \cite{khamsi}}\label{kham}
Let $\rho$ be a  function modular  satisfying $\Delta_2$-condition
and $G$ be  a $\|\cdot\|_\rho$-closed subset of
$\mathcal{X}_\rho$. If $T:G\rightarrow G$ is a mapping satisfying
\begin{equation}
\exists \,c\in [0,1) : \rho(Tx-Ty)\leq c\rho(x-y)\,\,\,\,\,\,\,\,\
(x,y\in G),
\end{equation}
and $\sup_n \rho(2T^n x)<\infty$ for
some $x\in G$, then $T$ has a fixed point.

\end{thm}

\begin{thm}{\rm \cite{han}}\label{hane}
Let $\mathcal{X}_\rho$ be a $\rho$-complete modular space. Suppose further that $\rho$  is an  $s$-convex modular  satisfying   $\Delta_2$-condition and has the
Fatou property. If $G$ is $\rho$-closed in $\mathcal{X}_\rho$ and
$T:G\rightarrow G$ is a mapping satisfying
\begin{equation}
\rho(c(Tx-Ty))\leq k^s\rho(x-y)\,\,\,\,\,\,\,\,\ (x,y\in G)
\end{equation}
for some $c,k\in \mathbb{R}^+$ with $c>\max\{1,k\}$, then $T$ has
a fixed point.
\end{thm}

In this note,  via an order theoretic approach, the modular version of  Banach contraction principle is
 proved. The result also answers, in a more general form,  an open question posed in \cite{han}.

We commence some basic notions of modular spaces used in theorems above. For more
details, the reader is asked to refer to \cite{koz, mus}. Some recent works on modular spaces may also be found in \cite{k1, k2, k3, k4}.

\begin{defin}\label{dov}{\rm A modular on a linear space $\mathcal{X}$ is a  functional  $\rho:\mathcal{X}\rightarrow [0,\infty]$ satisfying  the  conditions:\\
\indent 1) $\rho(x)=0$ if and only if $x=0$,\\
\indent 2) $\rho(x)=\rho(-x)$,\\
\indent 3) $\rho(\alpha x+\beta y)\leq \rho(x)+\rho(y)$, for all
$x,y\in
\mathcal{X}$ and $\alpha, \beta \geq 0$, $\alpha +\beta=1$. \\
\noindent Then, the  vector subspace
$$\mathcal{X}_\rho = \{x\in X: \rho(\alpha x)\rightarrow 0
\,\,\,\rm{as} \,\,\,\alpha\rightarrow0\},$$
 of
$\mathcal{X}$ is called a modular space. \noindent If the
condition (3) is replaced by
$$\rho(ax+by)\leq a^s\rho(x)+b^s\rho(y)$$ for all $x,y\in \mathcal{X}$ and
all $a,b\geq0$ satisfying $a^s+b^s=1$, where $s\in(0,1]$, then the
modular $\rho$ is called an $s$-convex modular on $\mathcal{X}$.}
\end{defin}

For every modular space $\mathcal{X}_\rho$ an $F$-norm $\|\cdot\|_\rho$ can be associated as: $$\|x\|_\rho=\inf \{t>0:\rho(t^{-1}x)\leq t\}, \,\,\,\,\,\,\,\,\, (x\in \mathcal{X}_\rho).$$
\begin{defin}\label{dovom}{\rm  Let $\mathcal{X}_{\rho}$ be a modular space. \\
\noindent (a) A sequence  $\{x_n\}_{n=1}^\infty$ in modular space
$\mathcal{X}_{\rho}$ is called\\
\indent (i) $\rho$-convergent to $x\in
\mathcal{X}_{\rho}$  if $\rho(x_n-x)\rightarrow 0$ as $n\rightarrow \infty$;\\ \indent (ii) $\rho$-Cauchy  if $\rho(x_m-x_n)\rightarrow 0$ as $m,n\rightarrow \infty$;\\
\noindent (b) $\mathcal{\mathcal{X}_\rho}$
is $\rho$-complete if every $\rho$-Cauchy sequence  in
$\mathcal{X}_\rho$ is $\rho$-convergent to a point of $\mathcal{X}_\rho$;\\
(c) A subset $C$ of $\mathcal{X}_\rho$ is said to be $\rho$-closed if it contains the $\rho$-limit of all
its $\rho$-convergent sequences;\\
\noindent (d)  $\rho$ is said to satisfy the $\Delta_2$-condition
if for each sequence $\{x_n\}_{n=1}^\infty$ in $\mathcal{X}_{\rho}$,
 $\rho(x_n)\rightarrow 0$, as $n\rightarrow\infty$ implies that $\rho(2x_n)\rightarrow 0$, as $n\rightarrow\infty$;\\
\noindent (e) The modular $\rho$ has the Fatou property if $$\rho(x-y)\leq \liminf_{n\rightarrow \infty} \rho(x_n-y_n),$$
whenever $\rho(x_n-x)\rightarrow 0$ and $\rho(y_n-y)\rightarrow 0$ as $n\rightarrow\infty$. }
\end{defin}

\section{Banach Contraction Principle}
%


\begin{thm}\label{BCPM}
Let $\mathcal{X}_\rho$ be a  complete modular space. If $\rho$ has the Fatou property,
$T:\mathcal{X}_\rho\rightarrow \mathcal{X}_\rho$ is a mapping
satisfying
\begin{equation}\label{aval}
\exists \,c\in [0,1) : \rho(Tx-Ty)\leq c\rho(x-y)\,\,\,\,\,\,\,\,\
(x,y\in \mathcal{X}_\rho),
\end{equation}
and $\sup_{n\geq 1}\rho(2T^n\omega)<\infty$, for some $\omega\in \mathcal{X}_\rho$, then $T$ has a fixed point.
\end{thm}
\begin{proof}
   Consider the family $\tau$ consisting of all subsets  $\mathcal{M}$ of $\mathcal{X}_\rho\times [0,\infty)$  such that $$\exists \, (b,\beta)\in \mathcal{M}\, ,\forall \,n\in \mathbb{N}:\,\,\,\, (T^nb,c^n\beta)\in \mathcal{M}, $$
 and
 $$\forall \,(x,\alpha), (y,\beta)\in \mathcal{M}:\, \rho(x-y)\leq |\alpha-\beta|.$$
   $\tau$ is nonempty. In fact, if $\omega$ satisfies $\sup_{n\geq 1}\rho(2T^n\omega)<\infty$, then $$\sup_{n\geq1}\rho(\omega-T^n\omega)<\infty,$$ and  we may choose $\alpha>0$ such that
 $\rho(\omega-T^n\omega)\leq \alpha-c^n\alpha$, for each $n\geq1$. Then, because of the assumption (\ref{aval}) the set consisting of all $(T^n\omega,c^n\alpha)$, $n\geq 0$ belongs
 to $\tau$. It is clear, by Zorn's lemma, that the family $\tau$  has a maximal element $\mathcal{P}$ with respect to partial order $\subseteq$ in $\tau$.
Define the binary relation $\preceq$ in the set
$\mathcal{P}$ by
$$(x,\alpha)\preceq
(y,\beta) \,\,\,\ \mbox{iff} \,\,\,\  \rho(x-y)\leq \alpha -\beta,$$
 where
$x,y\in \mathcal{X}_\rho$ and $\alpha, \beta \in  [0,\infty)$.
Then, $(\mathcal{P},\preceq)$ is a totally ordered set. We show that  $\mathcal{P}$  has a
 maximum element. If $p\in \mathcal{P}$, then there exist   $x_p\in \mathcal{X}_\rho$ and $\alpha_p\in  [0,\infty)$ such that $p=(x_p,\alpha_p)$. Note that  the  net $\{\Lambda_p\}_{p\in \mathcal{P}}$ is increasing, where $\Lambda_p=(x_p,\alpha_p)$.  That is,
\begin{equation} \label{ca} p\preceq q \Rightarrow  \rho(x_p-x_q)\leq \alpha_p -\alpha_q.
\end{equation}
Hence, the net $\{\alpha_p\}_{p\in \Gamma}$ is  decreasing in
$[0,\infty)$. Replace this net by the sequence $\{\alpha_n\}_{n=1}^\infty$  and assume that $\alpha
_n\rightarrow \alpha_0$ as $n\rightarrow \infty$. The inequality (\ref{ca}) shows that the
corresponding sequence $\{x_n\}_{n=1}^{\infty}$ is a $\rho$-Cauchy sequence  in
$\mathcal{X}_\rho$.  Suppose that
$\{x_n\}_{n=1}^{\infty}$ is $\rho$-convergent to $x_0\in
\mathcal{X}_\rho$. Again, from (\ref{ca})  and the Fatou property it follows that
  $$ \forall p\in \mathcal{P}:\,\, \rho(x_p-x_0)\leq \liminf_{m\rightarrow\infty} \rho(x_p-x_m)\leq \alpha_p-\alpha_0.$$
This implies that $(x_p,\alpha_p)\preceq (x_0,\alpha_0)$, for each $p\in \mathcal{P}$ and therefore $(x_0,\alpha_0)$
is the maximum of  $\mathcal{P}$, since $\mathcal{P}$ is maximal.
\\
 Now, let $(b,\beta)$ be an element of $\mathcal{P}$ such that $(T^nb,c^n\beta)\in \mathcal{P}$, for each $n\in \mathbb{N}$. Then,
 $(T^nb,c^n\beta)\preceq (x_0,\alpha_0)$, for each $n\in \mathbb{N}$ implies that $\alpha_0=0$,
 and therefore $\rho(T^nb-x_0)\rightarrow 0$, as $n\rightarrow \infty$. This also
 implies that $\rho(T^{n+1}b-Tx_0)\rightarrow 0$, as $n\rightarrow \infty$. Hence, $Tx_0=x_0$.
%
\end{proof}

\noindent {\bf{Remark}}. A modular $\rho$ is said to satisfy the $\Delta_2$-type
condition if there exists  $\alpha>0$ such that $\rho(2x)\leq
\alpha\rho(x)$ for all $x\in \mathcal{X}_\rho$. The Theorem \ref{BCPM} gets a simpler proof if $\rho$ satisfies the $\Delta_2$-type
condition. To see this, let $k$ be the $\Delta_2$-type constant
of $\rho$. Without loss of generality, we may suppose that  $ck
<\frac{1}{2}$. Otherwise, choose the integer $n$  such that $c^nk
<\frac{1}{2}$ and then replace $T^n$ by $T$. Let $x\in
\mathcal{X}_\rho$; we show that $\{T^n x\}$ is a $\rho$-Cauchy sequence.
Let $\epsilon>0$ be given and $$W_\epsilon=\{(x,y)\in
\mathcal{X}_\rho \times \mathcal{X}_\rho : \rho(x-y)<\epsilon\}.$$
There exists $N\in \mathbb{N}$ such that
$$(T^{n-1}x,T^nx)\in W_\epsilon,$$ for each $n>N$. Fix $n>N$. It
suffices to show that
\begin{equation}\label{cau} (T^nx,T^{n+p}x)\in W_\epsilon,\,\,\,\,\,\,\,\,\,\,\,\,\,\,\,\,\,\,\,\,\,\,(p\in \mathbb{N}).
\end{equation}
By induction on $p$, suppose that (\ref{cau}) is valid for some
fixed $p$. Since $$(T^{n-1}x,T^nx)\in W_\epsilon,
\,\,\,\,\,\,\,\,\,\, (T^nx,T^{n+p}x)\in W_\epsilon$$ we get
$$\rho(T^{n-1}x-T^{n+p}x)\leq
ck\rho(T^{n-1}x-T^nx)+ck\rho(T^{n}x-T^{n+p}x)< \epsilon.$$
Therefore, $$\rho(T^{n}x-T^{n+p+1}x)\leq c\rho(T^{n-1}x-T^{n+p}x)<
\epsilon.$$ Hence, $\{T^n x\}$ is a $\rho$-Cauchy sequence and by
$\rho$-completeness must converge. Let $T^nx\rightarrow y\in
\mathcal{X}_\rho$, as $n\rightarrow\infty$. From (\ref{aval}), it
follows that $T^{n+1}x\rightarrow Ty$, as $n\rightarrow \infty$
and consequently the uniqueness of the $\rho$-limit implies that
$Ty=y$.

The approach given in the proof of Theorem 3 may also be seen in  \cite{fb}. The idea given in the remark above is, in fact, a simple use of uniform space techniques in
fixed point theory (see e.g., \cite{kn1, kn2}).




\begin{thebibliography}{10}


\bibitem{kn1}  Aghanians, A.; Fallahi, K.; Nourouzi, K., {\it Fixed points for $G$-contractions on uniform spaces endowed with a graph}. Fixed Point Theory Appl.,  2012:182, 12 pages.

\bibitem{kn2}  Aghanians, A.; Fallahi, K.; Nourouzi, K., {\it An entourage approach to the contraction principle in uniform spaces endowed with a graph}. Panamer. Math. J.,  23 (2013), no.~2, 87--102.





\bibitem{han}  Ait Taleb, A.; Hanebaly, E., \textit{A fixed point theorem and its application to integral equations in
 modular function spaces}. Proc. Amer. Math. Soc. 128 (2000), no. 2,
 419--426. 
 
 \bibitem{fb} Bazazi, F.; Nourouzi, K.; O'Regan, D., {\it Two order theoretic proofs for a fixed point theorem in partially ordered metric spaces and its application to trace class operators}. Fixed Point Theory.  (to appear)
 
 \bibitem{k1}  Fallahi, K.; Nourouzi, K., {\it Probabilistic modular spaces and linear operators}. Acta Appl. Math.,  105 (2009), no.~2, 123--140.




\bibitem{khamsi}  Khamsi, M. A.; Kozlowski, W. M.; Reich, S., \textit{Fixed point theory in
modular function spaces}. Nonlinear Anal. 14 (1990), no. 11,
935--953. MR1058415


\bibitem{k2}  Lael, F.; Nourouzi, K. {\it Fixed points of mappings defined on probabilistic modular spaces}. Bull. Math. Anal. Appl.,  4 (2012), no.~3, 23--28.


\bibitem{koz}  Kozlowski, Wojciech M., \textit{Modular function spaces}. Monographs and Textbooks in Pure
and Applied Mathematics, 122. Marcel Dekker, Inc., New York, 1988.
MR1474499

\bibitem{mus} Musielak, J., \textit{Orlicz spaces and modular spaces}. Lecture Notes in Mathematics,
1034. Springer-Verlag, Berlin, 1983. 

\bibitem{k3}  Nourouzi, K.; Shabanian, S., {\it Operators defined on $n$-modular spaces}. Mediterr. J. Math.,  6 (2009), no.~4, 431--446.


\bibitem{k4}    Nourouzi, K., {\it Probabilistic modular spaces, WSPC Proceedings}. December (2008), 814- 818.




    
 \bibitem{k5}    Nourouzi, K., {\it Baire's Theorem in  probabilistic modular spaces}. Lecture Notes in Engineering and Computer Science, WCE 2008, Vol. II, 916-917.


\end{thebibliography}
\end{document}